\documentclass[reqno]{amsart}
\usepackage{amssymb}
\usepackage{mathrsfs}
\usepackage{amsmath,amstext,amsfonts,verbatim}

\usepackage[latin1]{inputenc}

\usepackage{amsmath,amsthm,amssymb}

\usepackage{amsfonts}

\usepackage{amsmath,amssymb}
\usepackage{graphicx}
\usepackage{color}
\usepackage{indentfirst}

\newtheorem{theorem}{Theorem}[section]
\newtheorem{lemma}[theorem]{Lemma}

\theoremstyle{definition}
\newtheorem{definition}[theorem]{Definition}

\newtheorem{prop}{Proposition}[section]

\theoremstyle{remark}

\numberwithin{equation}{section}

\newcommand{\neweq}[1]{\begin{equation}\label{#1}}

\def\phi{\varphi}

\def\incep{\left\{\begin{array}{cl} }
 \def\termin{\end{array}\right. }
\def\2af{2^*_\alpha}

\begin{document}

\title [A variational principle for entropy of a random dynamical system]{\textbf{A variational principle for entropy of a random dynamical system}}

\author{Yuan Lian}
\address{College \  of \ Mathematics \  and \ Statistics\ , Taiyuan Normal University, Taiyuan, 030619, China}
\email{andrea@tynu.edu.cn}



\keywords{Topological entropy, separated set, random dynamical system}


\date{}
\subjclass[2010]{37A35,37B40.}

\begin{abstract}
In this article, I give a definition of topological entropy for random dynamical systems associated to an infinite countable discrete amenable group action. I obtain a variational principle between the topological entropy and measurable fiber entropy of a random dynamical system.
\end{abstract}

\maketitle

\section{introduction}
Entropy is a widely used concept of chaoticity that is related to dynamical systems. In the concepts of entropy, there are two classical concepts: one is the pure topological concept called topological entropy, and the other is the measure-theoretical concept called measure-theoretical entropy. Firstly, the topological entropy of the dynamical system is defined using all finite open covers of $X$, and then it is calculated as the supremum on all these finite open covers.In a similar way, the measure theory entropy is defined using covers, with respect to the $T$ invariant Borel probability measure. For any finite measurable partition, it is then calculated as the supremum on all measurable partitions. Due to the introduction of the measure theory entropy of invariant Borel probability measure \cite{K1} and topological entropy \cite{AKM}, the relationship between these two types of entropy has become the concern of many scholars, so the relation between topological entropy and measure-theoretic entropy has been established, that is, the classical variational principle has emerged \cite{G,G2,P}. The topological variational principle proves that the supremum of the measure-theoretical entropy of all $T$-invariant Borel probability measures is the topological entropy for any topological dynamical system.

One of the main problems in the earliest process of proving the variational principle of topological entropy was the direct comparison between the topological entropy of open cover and the measure-theoretic entropy of measurable partitions. However, such comparisons can be avoided in some proof. For example, by considering alternative definitions of topological entropy that do not involve open cover (see \cite{DGS}). In addition, motivated by Huang and Lu's work, in this paper, I introduction a definition of topological entropy of a random dynamical system from a viewpoint of separated set and I prove a variational principle between the measure fiber entropy and the topological entropy that I presented in this paper of a random dynamical systems associated to an infinite countable discrete amenable group action (see \cite{HLZ,W}).

This paper is organized as follows. In Sec. 2, I recall the notions of a random dynamical system and its fiber entropy, introduce a topological entropy of separated set of the random dynamical system. In Sec. 3, I give a variational principle between the topological entropy and measurable fiber entropy.

\section{Preliminaries}
Let $G$ be an infinite countable discrete amenable group action, the knowledge of the measure fiber entropy of a CRDS associated to $G$, I cite the theory in literature \cite{DZ} and let's review the concepts that will be used later in this article.

I set $(\Omega,\mathcal{F},\mathbb{P})$ is complete and countably separated space. $(\Omega,\mathcal{F},\mathbb{P},G)$ denote an MDS(measurable dynamical G-system). Let $(X,\mathcal{B})$ be a measurable space, $\mathcal{E}\in\mathcal{F}\times\mathcal{B}$ and $(\mathcal{E},(\mathcal{F}\times\mathcal{B})_{\mathcal{E}})$ is a measurable space. I write $\mathcal{E}_{\omega}=\{x\in X:(\omega,x)\in \mathcal{E}\}$ for
each $\omega\in\Omega$. A continuous bundle random dynamical system(CRDS) associated to $(\Omega,\mathcal{F},\mathbb{P},G)$ is a family
$$\mathbf{F}=\{F_{g,\omega}:\mathcal{E}_{\omega}\rightarrow\mathcal{E}_{g\omega}\mid g\in G, \omega\in \Omega\}$$ satisfying:

\begin{enumerate}
\item let $e_{G}$ is the identity element of $G$, $F_{e_{G},\omega}$ is the identity over $\mathcal{E}_{\omega}$, for all $\omega\in\Omega$,

\item the map $(\mathcal{E},(\mathcal{F}\times\mathcal{B})_{\mathcal{E}})\rightarrow(X,\mathcal{B})$, given by $(\omega,x)\mapsto F_{g,\omega}(x)$,
      is measurable, for each $g\in G$,
\item let $\omega\in\Omega$ and $g_{1},g_{2}\in G,F_{g_{2},g_{1}\omega}\circ F_{g_{1},\omega}=F_{g_{2}g_{1},\omega}$,
\item if for $\mathbb{P}-$a.e. $\omega\in\Omega$, $\emptyset\neq\mathcal{E}_{\omega}\subseteq X$ is a compact subset and $F_{g,\omega}$ is a continuous map for
      any $g\in G$.
\end{enumerate}

In this case, the corresponding skew product transformation is $\Theta_{g}:(\omega,x)\mapsto(g\omega,F_{g,\omega}x)$ for each $g\in G$.

If $\Omega$ is a singleton, $(\Omega,\mathcal{F},\mathbb{P},G)$ is a trivial MDS. Let $X$ is a compact metric space, if there is an non-empty compact subset $K\subseteq X$ such that $(K,G)$ is a topological dynamical $G-$systems (TDS), that is, a CRDS associated to a trivial MDS.

I denote the space of all probability measures on $\Omega\times X$ having marginal $\mathbb{P}$ on $\Omega$ by $\mathcal{P}_{\mathbb{P}}(\Omega\times X)$. Put
$$
\mathcal{P}_{\mathbb{P}}(\mathcal{E})=\{\mu\in\mathcal{P}_{\mathbb{P}}(\Omega\times X):\mu(\mathcal{E})=1\}
$$
Denote by $\mathcal{P}_{\mathbb{P}}(\mathcal{E},G)$ the set of all $G-$invariant elements from $\mathcal{P}_{\mathbb{P}}(\mathcal{E})$. I use $\mathcal{E}_{\mathbb{P}}(\mathcal{E},G)$ to denote the set of the ergodic elements of $\mathcal{P}_{\mathbb{P}}(\mathcal{E})$.

Write $C_{\mathcal{E}}$ the set of all finite covers  of $\mathcal{E}$ and $P(\mathcal{E})$ finite partitions of $\mathcal{E}$. Let $\mathcal{U}\in C_{\mathcal{E}}$, for $\mu\in \mathcal{P}_{\mathbb{P}}(\mathcal{E},G)$, by \cite{DZ}, the $\mu-$fiber entropy of $\mathbf{F}$ with respect to $\mathcal{U}$ is
$$h^{(r)}_{\mu}(\mathbf{F},\mathcal{U})=\lim_{n\rightarrow\infty}\frac{1}{|F_{n}|}H_{\mu}(\mathcal{U}_{F_{n}}|\mathcal{F}_{\mathcal{E}})$$
and the $\mu-$fiber entropy is $$h^{(r)}_{\mu}(\mathbf{F})=\sup_{\alpha\in P(\mathcal{E})}h^{(r)}_{\mu}(\mathbf{F},\alpha).$$

\begin{definition}
For each nonempty finite subset $F$ of $G$. Define $d^{\omega}_{F}$ on $\mathcal{E}_{\omega}$ by the following formula
$$
d^{\omega}_{F}(x,y)=\max_{s\in F}d(F_{s,\omega}x,F_{s,\omega}y)\quad\text{for}\quad x,y\in \mathcal{E}_{\omega}.
$$
\end{definition}

\begin{definition}
Let $F$ be a nonempty finite subset of $G$. A set $E\subseteq \mathcal{E}_{\omega}$ is said to be $(\omega,F,\varepsilon,\mathbf{F})-$separated if $x,y\in E, x\neq y$ implies
$d^{\omega}_{F}(x,y)>\varepsilon$.
\end{definition}

\begin{definition}
Let $F$ be a nonempty finite subset of $G$ and $\varepsilon>0$, I write $Sep(\omega,F,\varepsilon,\mathbf{F})$ for the maximum cardinality of a $(\omega,F,\varepsilon,\mathbf{F})-$separated subset of $\mathcal{E}_{\omega}$.
\end{definition}

Let's prove the following lemma before proving the main result.

\begin{lemma}
Let $\mathbf{F}$ be a RDS associated to $G-$action. For any $F\in \mathcal{F}(G)$ and a positive real number $\varepsilon$, the function $Sep(\omega ,F,\varepsilon,\mathbf{F})$ is measurable in $\omega$, and for each $\delta>0$ there exists a family of maximal $(\omega,F,\varepsilon,\mathbf{F})$ separated sets $L_{\omega}\subseteq \mathcal{E}_{\omega}$, satisfying
\begin{equation}
\sharp(L_{\omega})\geq(1-\delta)Sep(\omega ,F,\varepsilon,\mathbf{F})
\end{equation}
where $\sharp(L_{\omega})$ denote the cardinality of $G_{\omega}$ and measurable in $\omega$ (in the sense that $G=\{(\omega,x):x\in \mathcal{E}_{\omega}\}\in \mathcal{F}\times \mathfrak{B}$). Moreover, the maximum in the definition of $Sep(\omega ,F,\varepsilon,\mathbf{F})$ can be taken only over measurable in $\omega$ families of $(\omega,F,\varepsilon,\mathbf{F})$ separated sets.
\end{lemma}

\begin{proof}
Let $F$ be a nonempty finite subset of $G$. $l\in N_{+}$, I will make the following five collection symbols for the sake of convenience.
$$H_{l}=\{(\omega,x_{1},x_{2},\ldots,x_{l}):\omega\in\Omega,x_{m}\in \mathcal{E}_{\omega},\forall m\},$$
$$M^{F}_{l}=\{(\omega,x_{1},x_{2},\ldots,x_{l})\in H_{l}:d^{\omega}_{F}(x_{m},x_{n})>\varepsilon,\forall m\neq n\},$$
$$M^{F,i}_{l}=\{(\omega,x_{1},x_{2},\ldots,x_{l})\in H_{l}:d^{\omega}_{F}(x_{m},x_{n})\geq(1+i^{-1})\varepsilon,\forall m\neq n\},$$
$$M^{F}_{l}(\omega)=\{(x_{1},x_{2},\ldots,x_{l}):(\omega,x_{1},x_{2},\ldots,x_{l})\in M^{F}_{l}\},$$
$$M^{F,i}_{l}(\omega)=\{(x_{1},x_{2},\ldots,x_{l}):(\omega,x_{1},x_{2},\ldots,x_{l})\in M^{F,i}_{l}\},$$

$H_{l}\in \mathcal{F}\times \mathfrak{B}^{l}$ can be known from the structure of $H_{l}$ by Theorem III.30 in \cite{CV}, where $\mathfrak{B}^{l}$ is the product $\sigma-$ algebra on the product space $X^{l}$. Note that
$$
d_{l}((x_{1},x_{2},\ldots,x_{l}),(y_{1},y_{2},\ldots,y_{l}))=\sum^{l}_{m=1}d(x_{m},y_{m})
$$
is the distance function on $X^{l}$, and if $\mathcal{E}_{\omega}^{l}$ denotes the product of $l$ copies of $\mathcal{E}_{\omega}$, then
$$d_{l}((x_{1},x_{2},\ldots,x_{l}), \mathcal{E}_{\omega}^{l})=\sum^{l}_{m=1}d(x_{m},\mathcal{E}_{\omega}),$$
is measurable in $\omega$ for each $(x_{1},x_{2},\ldots,x_{l})\in X^{l}$. Next, I will define $l(l-1)/2$ measurable functions $\psi_{mn}, 1\leq m<n\leq l$, on $H_{l}$ by $$\psi_{mn}(\omega,x_{1},x_{2},\ldots,x_{l})=d^{\omega}_{F}(x_{m},x_{n})(\varepsilon^{-1}).$$
Then
$$
M^{F}_{l,i}=\bigcap_{1\leq m<n\leq l} \psi^{-1}_{mn}[1+i^{-1},\infty)\in\mathcal{F}\times \mathfrak{B}^{l}.
$$
It is easy to see that each $M^{F,i}_{l}(\omega)$ is a closed subset of $\mathcal{E}_{\omega}^{l}$, by continuity of the RDS $\mathbf{F}$, and so it is compact. Clearly, $M^{F,i}_{l}\uparrow M^{F}_{l}$ and $M^{F,i}_{l}(\omega)\uparrow M^{F}_{l}(\omega)$ as $i\uparrow \infty$, in particular, $M^{F}_{l}\in \mathcal{F}\times \mathfrak{B}^{l}$.

Set $t_{F,i}(\omega,\varepsilon)=max\{l:M^{F,i}_{l}(\omega)\neq\emptyset\}$. By Lemma 4.1 in \cite{DZ}, it follows that
$$
\{\omega:t_{F,i}(\omega,\varepsilon)\geq l\}=\{\omega:M^{F,i}_{l}(\omega)\neq\emptyset\}=Pr_{\Omega}M^{F,i}_{l}\in \mathcal{F},
$$
where $Pr_{\Omega}$ is the projection of $\Omega\times X^{l}$ to $\Omega$, and so $t_{F,i}(\omega,\varepsilon)$ is measurable in $\omega$. In addition, $Sep(\omega ,F,\varepsilon,\mathbf{F})$ is also measurable, this since

\begin{equation*}
\begin{split}
\{\omega:Sep(\omega ,F,\varepsilon,\mathbf{F})\geq l\}=\{\omega:M^{F}_{l}(\omega)\neq\emptyset\}&=\bigcup_{j=1}^{\infty}\bigcap_{i=j}^{\infty}\{\omega:M^{F,i}_{l}(\omega)\neq\emptyset\}\\
&=\bigcup_{j=1}^{\infty}\bigcap_{i=j}^{\infty}\{\omega:t_{F,i}(\omega,\varepsilon)\geq l\}\in \mathcal{F}.
\end{split}
\end{equation*}

For any positive real number $\delta>0$, write
$$
N^{F}_{l,\delta}=\{(\omega,x_{1},x_{2},\ldots,x_{l})\in H_{l}:l\geq(1-\delta)Sep(\omega ,F,\varepsilon,\mathbf{F})\}
$$
and $N^{F}_{l,\delta}$ is an element of $\mathcal{F}\times \mathfrak{B}^{l}$. So I get the following results
$$
L^{F}_{l,\delta}=N^{F}_{l,\delta}\cap M^{F}_{l}\in\mathcal{F}\times \mathfrak{B}^{l}
$$
and
$$
L^{F,i}_{l,\delta}=N^{F}_{l,\delta}\cap M^{F,i}_{l}\in\mathcal{F}\times \mathfrak{B}^{l}.
$$

Now assuming the following sets
$$
L^{F}_{l,\delta}(\omega)=\{(x_{1},x_{2},\ldots,x_{l}):(\omega,x_{1},x_{2},\ldots,x_{l})\in L^{F}_{l,\delta}\}
$$
$$
L^{F,i}_{l,\delta}(\omega)=\{(x_{1},x_{2},\ldots,x_{l}):(\omega,x_{1},x_{2},\ldots,x_{l})\in L^{F,i}_{l,\delta}\}.
$$
I know that $L^{F,i}_{l,\delta}(\omega)$ are compact sets and $L^{F,i}_{l,\delta}(\omega)\uparrow L^{F}_{l,\delta}(\omega)$ as $i\uparrow\infty$. I then know that
$$
\tilde{\Omega}_{l,i}=\{\omega:t_{F,i}(\omega,\varepsilon)=Sep(\omega ,F,\varepsilon,\mathbf{F})=l\}\cap\{\omega:L^{F,i}_{l,\delta}\neq\emptyset\}
$$
are measurable, and the sets $\Omega_{l,i}=\tilde{\Omega}_{l,i}\setminus\tilde{\Omega}_{l,i-1}, i=1,2,\ldots$ with $\tilde{\Omega}_{l,0}=\emptyset$ are measurable, disjoint and $\bigcup_{l,i\geq1}\Omega_{l,i}=\Omega$.

I know from see Theorem III.30 in \cite{CV} that the mulfunction $\Phi_{l,i,\delta}$ defined by $\Phi_{l,i,\delta}(\omega)=L^{F,i}_{l,\delta}(\omega)$ for $\omega\in \Omega_{l,i}$ is measurable, and there is a measurable selection $\tau_{l,i,\delta}$ which is a measurable map $\tau_{l,i,\delta}:\Omega_{l,i}\rightarrow X^{l}$ such that $\tau_{l,i,\delta}(\omega)\in L^{F,i}_{l,\delta}(\omega)$ for $\omega\in \Omega_{l,i}$. Let $\pi_{l}:X^{l}\rightarrow P(X)$ be the multifunction defined by
$$\pi_{l}(x_{1},x_{2},\ldots,x_{l})=\{x_{1},x_{2},\ldots,x_{l}\}\subseteq X,$$
where $P(X)$ represents a family of sets consisting of all subsets of $X$. Then $\pi_{l}\circ \tau_{l,i,\delta}$ is a multifunction assigning to each $\omega\in \Omega_{l,i}$ a maximal $(\omega ,F,\varepsilon,\mathbf{F})$ separated set $L_{\omega}$ in $\mathcal{E}_{\omega}$ for which (3.1) holds true.

Let $U$ be any open subset of $X$, denote $V^{l}_{U}(m)=\{(x_{1},x_{2},\ldots,x_{l})\in X^{l}:x_{m}\in U\}$ and it is an open subset of $X^{l}$. Then
$$
\{\omega\in \Omega_{l,l}:\pi_{l}\circ \tau_{l,i,\delta}(\omega)\cap U\neq\emptyset\}=\bigcup^{l}_{j=1}\tau^{-1}_{l,i,\delta}V^{l}_{U}(j)\in \mathcal{F}.
$$

For each $\omega\in \Omega_{l,i}$, define $j_{l}(\omega)=i$ and
$$
\varphi_{\delta}(\omega)=\pi_{Sep(\omega ,F,\varepsilon,\mathbf{F})}\circ\tau_{Sep(\omega ,F,\varepsilon,\mathbf{F}),j_{l}(\omega),\delta}(\omega)
$$
then
$$
\{\omega:\varphi_{\delta}(\omega)\cap U\neq\emptyset\}=\bigcup_{l,i=1}^{\infty}(\Omega_{l,i}\cap\{\omega:\zeta_{l}\circ \tau_{l,i,\delta}(\omega)\cap U\neq\emptyset\})\in \mathcal{F}.
$$
Hence $\varphi_{\delta}$ is a measurable multifunction, for each $\omega\in\Omega$ a maximal $(\omega ,F,\varepsilon,\mathbf{F})$ separated set $L_{\omega}$ for which (2.1) holds true since Theorem III.30 in \cite{CV} and Lemma 2.4 follows this since $\delta>0$ is arbitrary.
\end{proof}

Using the above lemma, I present the following definition of topological entropy.

\begin{definition}
The topological entropy of random dynamical system $\mathbf{F}$ is given by
$$
h_{top}(\mathbf{F},\mathcal{E})=\lim_{\varepsilon\rightarrow0}h^{\varepsilon}_{top}(\mathbf{F},\mathcal{E}).
$$
where
$$h^{\varepsilon}_{top}(\mathbf{F},\mathcal{E})=\limsup_{n\rightarrow\infty}\frac{1}{|F_{n}|}\int_{\Omega}\log Sep(\omega,F_{n},\varepsilon,\mathbf{F})dP(\omega).$$
\end{definition}

It is easy to show the following proposition (see\cite{DZ,K,Ki1,KL17})

\begin{prop}
Let $\mathbf{F}$ be a continuous RDS on Polish
space $X$ with Borel $\sigma$-algebra $B_{X}$ over the Polish
system $(\Omega,\mathcal{F},\mathbb{P},G)$. Then
$$h_{top}(\mathbf{F},\mathcal{E})=h^{(r)}_{top}(\mathbf{F})$$
where $h^{(r)}_{top}(\mathbf{F})$ is the fiber topological entropy of $\mathcal{F}$.

\end{prop}
\section{Main result}

Using the approach of \cite{K} and Lemma 2.4, similar to Proposition 3.7 in \cite{HL}, I have the following result.
\begin{prop}(Variational principle). Let $\mathbf{F}$ be a continuous RDS on Polish
space $X$ with Borel $\sigma$-algebra $B_{X}$ over the Polish
system $(\Omega,\mathcal{F},\mathbb{P},G)$. Then
$$h_{top} (\mathbf{F},\mathcal{E})=sup\{h^{(r)}_{\mu}(\mathbf{F}):\mu\in \mathcal{P}_{\mathbb{P}}(\mathcal{E},G) \}.$$
If in addition $(\Omega,\mathcal{F},\mathbb{P},G)$ is ergodic,
then
$$h_{top} (\mathbf{F},\mathcal{E})=sup\{h^{(r)}_{\mu}(\mathbf{F}):\mu\in \mathcal{E}_{\mathbb{P}}(\mathcal{E},G) \}.$$
\end{prop}

\begin{proof}
I will prove it in the following two steps.

\medskip\noindent{\bf{Step 1.}}\quad {\it {I will show that $h^{(r)}_{\mu}(\mathbf{F})\leq h_{top} (\mathbf{F},\mathcal{E})$ for all $\mu\in \mathcal{P}_{\mathbb{P}}(\mathcal{E},G)$.}}

By the proof of Theorem 4.6 of \cite{DZ}, I only need to prove
\begin{equation}
 h^{(r)}_{\mu}(\mathbf{F},(\Omega\times \alpha)_{\mathcal{E}})\leq h_{top}(\mathbf{F},\mathcal{E}), \,\forall\, \alpha\in P(X),  and \ \mu\in \mathcal{P}_{\mathbb{P}}(\mathcal{E},G).
\end{equation}

Let $\mu\in \mathcal{P}_{\mathbb{P}}(\mathcal{E},G)$, any finite measurable partition $\alpha=\{A_{1},A_{2},\ldots,A_{l}\}$ of $X$ and take any positive real number $\varepsilon$ with $\varepsilon l\log l<1$. Write $$\alpha(\omega)=\{A_{1}(\omega),A_{2}(\omega),\ldots,A_{l}(\omega)\},$$ and $$A_{m}(\omega)=A_{m}\cap \mathcal{E}_{\omega},m=1,2,\ldots,l,$$
then $\alpha(\omega)$ is the corresponding partition of $\mathcal{E}_{\omega}$. From the regularity of the measure $\mu$, I can find a family of compact sets that satisfies the following conditions
$$C_{m}\subseteq A_{m},m=1,2,\ldots,l,$$ and
\begin{equation}
 \mu(A_{m}\backslash C_{m})=\int\mu_{\omega}(A_{m}(\omega)\backslash C_{m}(\omega))d\mathbb{P}(\omega)<\varepsilon.
\end{equation}
where $C_{m}(\omega)=C_{m}\cap \mathcal{E}_{\omega}$. Set $C_{0}(\omega)=\mathcal{E}_{\omega}\backslash \bigcup^{l}_{m=1}C_{i}(\omega)$, then
\begin{equation*}
\begin{split}
\int\mu_{\omega}(C_{0}(\omega))d\mathbb{p}(\omega)&=\int\mu_{\omega}(\mathcal{E}_{\omega}\backslash \bigcup^{l}_{m=1}C_{i}(\omega))d\mathbb{P}(\omega)\\
&=\int\mu_{\omega}\big(\bigcup^{l}_{m=1}A_{m}(\omega)\backslash \bigcup^{l}_{m=1}C_{i}(\omega)\big)d\mathbb{P}(\omega)\\
&\leq\int\sum^{l}_{i=1}\mu_{\omega}(A_{m}(\omega)\backslash C_{m}(\omega))d\mathbb{P}(\omega)\\
&<k\varepsilon.
\end{split}
\end{equation*}

From the above I can see that $C(\omega)=\{C_{0}(\omega),C_{1}(\omega),C_{2}(\omega),\ldots,C_{l}(\omega)\}$ is a partition of $\mathcal{E}_{\omega}$, and the partition satisfies the inequality
\begin{equation*}
\begin{split}
H_{\mu_{\omega}}(A(\omega)|C(\omega)))&=-\sum^{l}_{m=0}\sum^{l}_{n=1}\mu_{\omega}(C_{m}(\omega))\frac{\mu_{\omega}(C_{m}(\omega)\cap A_{n}(\omega))}{\mu_{\omega}(C_{m}(\omega))}\log\frac{\mu_{\omega}(C_{m}(\omega)\cap A_{n}(\omega))}{\mu_{\omega}(C_{m}(\omega))}\\
&=-\mu_{\omega}(C_{0}(\omega))\sum^{l}_{n=1}\frac{\mu_{\omega}(C_{0}(\omega)\cap A_{n}(\omega))}{\mu_{\omega}(C_{0}(\omega))}\log\frac{\mu_{\omega}(C_{0}(\omega)\cap A_{n}(\omega))}{\mu_{\omega}(C_{0}(\omega))}\\
&\leq\mu_{\omega}(C_{0}(\omega))\log l.
\end{split}
\end{equation*}
hence
$$
\int H_{\mu_{\omega}}(A(\omega)|C(\omega))d\mathbb{P}(\omega)\leq\int \mu_{\omega}(C_{0}(\omega))\log ld\mathbb{P}(\omega)\leq l\varepsilon \log l<1.
$$

By Proposition 4.5 in \cite{DZ}, that ia, a propertie of conditional entropy,
\begin{equation}
h^{(r)}_{\mu}(\mathbf{F},(\Omega\times \alpha)_{\mathcal{E}})\leq h^{(r)}_{\mu}(\mathbf{F},(\Omega\times C)_{\mathcal{E}}).
\end{equation}
where I denote $$(\Omega\times \alpha)_{\mathcal{E}}=\{(\Omega\times A_{m})\cap \mathcal{E}:A_{m}\in \alpha\},$$
$$(\Omega\times C)_{\mathcal{E}}=\{(\Omega\times C_{m})\cap \mathcal{E}:C_{m}\in C\}.$$

Set $C_{F}(\omega)=\bigvee_{g\in F}F^{-1}_{g,\omega}C(g\omega)$. Then I can get
$$H_{\mu_{\omega}}(C_{F}(\omega))=\sum_{B\in C_{F}(\omega)}\mu_{\omega}(B)(-\log \mu_{\omega}(B))\leq \log \sharp(C_{F}(\omega)).$$
by the standard inequalitty
$$\sum_{1\leq m\leq k}p_{m}(a_{m}-\log p_{m})\leq \log \sum_{1\leq m\leq k} e^{a_{m}}.$$
for any probability vector $(p_{1},p_{2},\ldots,p_{k})$.

Pick small positive real number $\delta>0$ and let $E$ be a maximal $(\omega ,F_{n},\delta,\mathbf{F})$ separated set of $\mathcal{E}_{\omega}$. Then for each $x\in \mathcal{E}_{\omega}$ there is an $y(x)\in E$ such that $d^{\omega}_{F_{n}}(x,y)<\delta$ by the proposition of maximal $(\omega ,F_{n},\delta,\mathbf{F})$ separated set $E$. I may assume that $\delta<min\{diam(A):A\in C_{F_{n}}(\omega)\}$ if necessary. It follows that
$$
H_{\mu_{\omega}}(C_{F_{n}}(\omega))\leq \log \sharp (C_{F_{n}}(\omega))\leq \log Sep(\omega ,F_{n},\delta,\mathbf{F}).
$$

Integrating this against $\mathbb{P}$, dividing by $|F_{n}|$, I derive
$$
\frac{1}{|F_{n}|}\int H_{\mu_{\omega}}(C_{F_{n}}(\omega))d\mathbb{P}(\omega)\leq\frac{1}{|F_{n}|}\int \log Sep(\omega,F_{n},\delta,\mathbf{F})d\mathbb{P}(\omega).
$$
and letting $n\rightarrow\infty$, I have
$$
h^{(r)}_{\mu}(G,(\Omega\times C)_{\mathcal{E}})\leq h^{\delta}_{top}(\mathbf{F},\mathcal{E}).
$$

Using (3.3), the inequality
$$h^{(r)}_{\mu}(G,(\Omega\times \alpha)_{\mathcal{E}})\leq h^{(r)}_{\mu}(G,(\Omega\times C)_{\mathcal{E}})\leq h^{\delta}_{top}(\mathbf{F},\mathcal{E}).$$
Since $\alpha$ and $\delta$ are arbitrary, it follows that
$$h^{(r)}_{\mu}(\mathbf{F})\leq h_{top}(\mathbf{F},\mathcal{E}).$$

\medskip\noindent{\bf{Step 2.}}\quad {\it {I will show that $h_{top} (\mathbf{F},\mathcal{E})\leq sup\{h_{\mu}(\mathbf{F}):\mu\in \mathcal{P}_{\mathbb{P}}(\mathcal{E},G) \}$.}}

Let $\{F_{n}\}$ be a F{\o}lner sequence for $G$, $\varepsilon>0$ is a small constant and employ Lemma 1.2 to choose a family of maximal $(\omega ,F_{n},\delta,\mathbf{F})$ separated sets that are measurable in $\omega$ such that
\begin{equation}
\sharp L(\omega ,F_{n},\delta,\mathbf{F})\geq\frac{1}{e}Sep (\omega,F_{n},\delta,\mathbf{F})
\end{equation}

For the sake of further proof, I need to define a probability measures $\nu^{(n)}$ on $\mathcal{E}$ in the following way
$$
\nu^{(n)}_{\omega}=\frac{\sum_{x\in L(\omega ,F_{n},\varepsilon,\mathbf{F})}\delta_{x}}{\sharp L(\omega ,F_{n},\varepsilon,\mathbf{F})}.
$$
So that $d\nu^{(n)}(\omega,x)=d\nu^{(n)}_{\omega}(x)d\mathbb{P}(\omega)$, and set
$$
\mu^{(n)}=\frac{1}{|F_{n}|}\sum_{g\in F_{n}}g\nu^{n}.
$$

By the definition of $h^{\varepsilon}_{top}(\mathbf{F},\mathcal{E})$ and by Proposition 4.3 in \cite{DZ}, I can choose a subsequence $\{n_{j}\}$ such that

\begin{equation}
lim_{j\rightarrow\infty}\frac{1}{|F_{n_{j}}|}\int log Sep(\omega ,F_{n},\varepsilon,\mathbf{F})dp(\omega)=h^{\varepsilon}_{top}(\mathbf{F},\mathcal{E})
\text{ and } \mu^{n_{j}}\Rightarrow \mu \ as \ j\rightarrow\infty,
\end{equation}
for some $\mu\in\mathcal{P}_{P}(\mathcal{E},G)$.

Further, I choose a partition $C=\{C_{1},\ldots,C_{k}\}$ of $X$ that satisfies $diam(C)\leq\varepsilon$ and $(\bigvee_{g\in F_{n}})g^{-1}C{(\omega)}$ is a clopen partition of $\mathcal{E}_{\omega}$ for $\mathbb{P}-$a.e. $\omega\in \Omega$. Set
$$
C(\omega)=\{C_{1}(\omega),\ldots,C_{k}(\omega)\}.
$$
I have by (3.4),
\begin{equation}
H_{\nu^{(n)}_{\omega}}(\bigvee_{g\in F_{n}}F_{g^{-1},g\omega}C(g\omega))=log\sharp L(\omega,F_{n},\varepsilon,\mathbf{F})\geq log Sep(\omega,F_{n},\varepsilon,\mathbf{F})-1,
\end{equation}
this since each element of $\bigvee_{g\in F_{n}}F_{g^{-1},g\omega}C(g\omega)$ contain at most one element of $L(\omega,F_{n},\varepsilon,\mathbf{F})$.

Let $P=\{P_{1},\ldots,P_{k}\},P=(\Omega\times C_{i})\cap \mathcal{E}$, then $P$ is a partition of $\mathcal{E}$ and
$$P_{i}(\omega)=\{x\in \mathcal{E}_{\omega}:(\omega,x)\in P_{i}\}=C_{i}(\omega).$$ Integrating in (3.6) against $\mathbb{P}$, I obtain by $(4.2)$ in \cite{DZ}, the inequality
\begin{equation}
H_{\nu^{(n)}}(\bigvee_{g\in F_{n}}\Theta_{g}^{-1}P|\mathcal{F}_{\mathcal{E}})\geq log Sep(\omega,F_{n},\varepsilon,\mathbf{F})-1
\end{equation}

Dividing by $|F_{n}|$ and then take the limit along a subsequence $n_{j}\rightarrow\infty$ satisfying (3.5) both sides of this inequality. Taking into account Proposition 4.3 in \cite{DZ}, in view of the choice of the partition $\alpha$, it follows that
$$
h^{\varepsilon}_{top}(\mathbf{F},\mathcal{E})\leq h^{(r)}_{\mu}(\mathbf{F})
$$
and letting $\varepsilon\rightarrow0$, the required inequality follows, completing the proof of Step 2.
\end{proof}



\bibliographystyle{amsplain}

\end{document}